\documentclass[graybox]{svmult}

\usepackage{type1cm}        % activate if the above 3 fonts are
\usepackage{makeidx}         % allows index generation
\usepackage{graphicx}        % standard LaTeX graphics tool
\usepackage{multicol}        % used for the two-column index
\usepackage[bottom]{footmisc}% places footnotes at page bottom

\usepackage{newtxtext}       %
\usepackage[varvw]{newtxmath}       % selects Times Roman as basic font

\usepackage[all,cmtip]{xy}
\makeindex             % used for the subject index
\newcommand{\N}{\mathbb N}
\newcommand{\R} {\mathbb R}

\newcommand{\resr}[1]{\raisebox{-1.2ex}{$\big| \raisebox{-0.3ex}{$#1$} $}}
 \newcommand{\sign}{\mbox{ sign}}

\newcommand{\lmg}{\Big\{ } \newcommand{\rmg}{\Big\} }
\newcommand{\meig}{\;\big| \;}
\newcommand{\menge}[2]{\lmg {#1}\meig{#2}\rmg}

\newcommand{\myref}[1]{{\rm (\ref{#1})}}

\newtheorem{thm}{Theorem}[section]
\newtheorem{cor}[thm]{Corollary}

\newtheorem{prop}[thm]{Proposition}

\begin{document}

\title*{Periodic Solutions of a Delay Differential Equation with a Periodic Multiplier}
% Use \titlerunning{Short Title} for an abbreviated version of
% your contribution title if the original one is too long
\author{Anatoli Ivanov%\orcidID{0000-0001-9652-4551} and
\\
Bernhard Lani-Wayda \\ %and \\
Sergiy Shelyag%\orcidID{0000-0002-6436-9347}
}
% Use \authorrunning{Short Title} for an abbreviated version of
% your contribution title if the original one is too long
\institute{A. Ivanov \at Pennsylvania State University, Dallas PA USA, \email{aivanov@psu.edu}
\and B. Lani-Wayda \at Justus-Liebig-Universit\"{a}t, Giessen, Germany, \email{Bernhard.Lani-Wayda@math.uni-giessen.de}
\and S. Shelyag \at Flinders University, Adelaide, Australia, \email{Sergiy.Shelyag@flinders.edu.au}
}
%
% Use the package "url.sty" to avoid
% problems with special characters
% used in your e-mail or web address
%
\maketitle

\abstract*{A simple non-autonomous scalar differential equation with delay, exponential decay,  nonlinear negative feedback and a 
periodic multiplicative coefficient is considered.  It is shown that stable slowly oscillating periodic solutions with the period of
the feedback coefficient, and also with the double period of the feedback coefficient exist.  The periodic solutions are built 
explicitly in the case of piecewise constant feedback function and the periodic  coefficient. The periodic dynamics are shown to 
persist under small perturbations of the equation which make it smooth. The results are confirmed and illustrated  by numerical 
simulations.}

\abstract{A simple non-autonomous scalar differential equation with delay, exponential decay,  nonlinear negative feedback and a 
periodic multiplicative coefficient   is considered.  It is shown that stable slowly oscillating periodic solutions with the period
of the feedback coefficient, and also with the double period of the feedback coefficient exist.  The periodic solutions are built 
explicitly in the case of piecewise constant feedback function and the periodic  coefficient. The periodic dynamics are shown to 
persist under small perturbations of the equation which make it smooth. The results are confirmed and illustrated  by numerical simulations.}

%%
%%
%% Section 1: Introduction
%%
%%
\section{Introduction}\label{Sect1}
The simple-looking  scalar delay differential equation
\begin{equation}\label{DDE0}
    x^\prime(t)=-\mu x(t)+f(x(t-\tau)),\quad f\in C^0(\mathbf{R}, \mathbf{R}),\; \mu\ge0,\; \tau>0
\end{equation}
is a paradigm for the rich and complex dynamics that such equations exhibit,  as well as for a wide range of applications in various fields of science, engineering, and other real world phenomena. Unlike the corresponding scalar   autonomous ordinary differential equation without  delay,  it can exhibit various dynamical behaviors such as sustained and periodic oscillations about equilibria, and complex (chaotic) dynamics. Some examples of these dynamics and further references can be found in e.g. \cite{AdHW,AdHM,GlaMac88,IvaSha91,L-W95,MacGla77,Pet83,Wal81b}. The broad spectrum of dynamical patterns and the intrinsic presence of delay in many applications make this equation a simple but adequate model to describe a range of various  processes in biological sciences. Among those are the well-known blood production and functioning models of Mackey-Glass \cite{MacGla77} and Lasota-Wazewska \cite{WazLas76}, population models of Nicholson's type \cite{GurBlyNis80,PerMalCou78}, other physiological and population dynamics models \cite{GlaMac88,Kua93}. The broad range of applications of this equation is described in multiple sources; here we mention a few books which focus on the applications as well as on the theoretical basics of such equations \cite{Ern09,Kua93,Smi11}. The fundamentals of the theory of this and more general delay differential equations can be found in e.g. monographs \cite{DieSvGSVLWal95,Hale77,HalSVL93}.

Extended models take  into account external influences --  in particular, periodic  factors that can be naturally present in some  components of the system,  
such as e.g. circadian rhythms, periodic medication intake, or  seasonal fluctuations in biological models. 
Such considerations  lead to non-autonomous delay differential equations with periodic time  dependence, and a special  case is the following  generalization of DDE (\ref{DDE0}):
\begin{equation}\label{DDE-per0}
    x^\prime(t)=-\mu(t) x(t)+a(t)f(x(t-\tau(t))),\quad f\in C^0(\mathbf{R}, \mathbf{R}),
\end{equation}
where all functions $\mu(t), a(t), \tau(t)$ are continuous and $T$-periodic. 
In this paper we focus on a special case of DDE (\ref{DDE-per0}) where the time delay $\tau$ and the decay rate $\mu$ are constants, while 
the delayed feedback input is periodic through the multiplicative coefficient $a(t)$ with the nonlinearity $f(x)$.

We note that there is a good number of publications  on scalar delay differential equations with periodic dependence on time in coefficients and delays. However, in most of them the time dependence is such that the corresponding equation does not possess a constant equilibrium any more; see e.g. \cite{Ai24,AmsIde13,Far17,FriWu92} and further references therein (our equation (\ref{DDE-per0}) always has the constant solution $x(t)\equiv0$). Some of those periodic equations do admit the constant (zero) solution; however, in those papers no distinction can be made between nontrivial periodic solutions and the existing zero solution; see e.g. \cite{LiKua01}.  Periodic solutions are frequently obtained from topological methods such as  Schauder's
fixed point theorem (see e.g. \cite{Far17}) or the 
Gaines-Mawhin coincidence degree (see e.g. \cite{LiKua01}), and do then not yield results on stability of the  periodic solutions. 
In previous papers, the periodic  influence  often appears in 
`additive'  form  like $... + g(t, x(t),x(t-r)$ (e.g., 
\cite{LiuYangGe05}), although 
multiplication with periodic coefficients has also been studied, e.g. in \cite{Far17} or \cite{JiangWeiZhang02}.  Recently, a partial case of equation (\ref{DDE-per0}), when $\mu\equiv0, \tau(t)=\tau_0 = $ constant,  was considered in \cite{IvaShe24,IvaShe23,IvaShe23-2}.

In the present paper we  focus on  a simple special  case of DDE (\ref{DDE-per0}) with the normalized delay $\tau=1$:
\begin{equation}\label{DDE-per}
    \dot x(t)=-\mu  x(t)+a(t)f(x(t-1)), 
\end{equation}
Here we can provide  detailed  information on the shape of solutions, 
criteria for existence of periodic motion, and their stability. 

In Section 2  we  explicitly  compute solutions for the case of piecewise constant $f$ and $a$, and derive conditions for existence and
stability/instability  of periodic solutions with the same period as the coefficient. In Section 3, a similar approach yields solutions 
with twice the coefficient period as their minimal period. 

Focusing on the single-period case, a smoothing procedure is described in Section 4. This  gives nearby equations within a  
$ C^1$-framework, preserving periodic solutions and their  stability or  instability.  
Such behavior is then robust under arbitrary  $C^1$-small perturbations.

%%
%% Section 2
%%
\section{Explicit Periodic Solutions}\label{Sect3}
In this section we build explicit forward (for $t\ge0$) solutions to equation (\ref{DDE-per}), where the piecewise constant functions $f=f_0(x)$ and $a=A_0(t)$ defined as follows:
$$
f_0(x)=-\text{sign} (x)=\begin{cases}
    -1,\,&\; \text{if}\: x>0\\
    0,\,&\; \text{if}\: x=0\\
    +1,\,&\; \text{if}\: x<0,
\end{cases}
$$
and 
\begin{equation}\label{A0}
A_0(t) = \begin{cases}
a_1>0,\;\text{if}\; 0\le t < p_1 \\
a_2>0,\;\text{if}\; p_1\le t < p_1+p_2:=T\\
\text{periodic extension outside}\; [0,T)\;\text{for all}\; t\in\mathbb{R}.
\end{cases}
\end{equation}.

In this section the initial functions which are sign definite on the initial 
time interval $[-1,0]$ will  be most important.

Note that the solution to the initial value problem $\dot x (t)=-\mu x(t)+A, x(t_0)=x_0, A$-const, is given by
$x(t)=A/\mu+(x_0-A/\mu)\exp\{-\mu (t-t_0)\}$.

%%
%% Subsection 3.1
%%
\subsection{Periodic solutions with the same period as the  coefficient $a$ }
\label{subsec3.1}

In this subsection we are dealing with slowly oscillating  solutions (i.e., 
the distance between consecutive zeroes is larger than the delay $1$) of the form as shown in Fig.~\ref{fig:single_period}. 
That is, for given $h>0,$ the first two zeros $t_1$ and $t_2$ of the solution satisfies $t_1+1<t_2<p_1$ and 
$x(t)<0,\forall t\in(t_1,t_2]$. In other words, the first negative semi-cycle of the solution $x=x(t,h)$ happens on the first constant 
segment $[0,p_1]$ of the coefficient $A(t)$. The solution remains positive on the interval $(t_2,T], T=p_1+p_2.$

\begin{figure}
    \centering
    \includegraphics[width=0.95\textwidth]{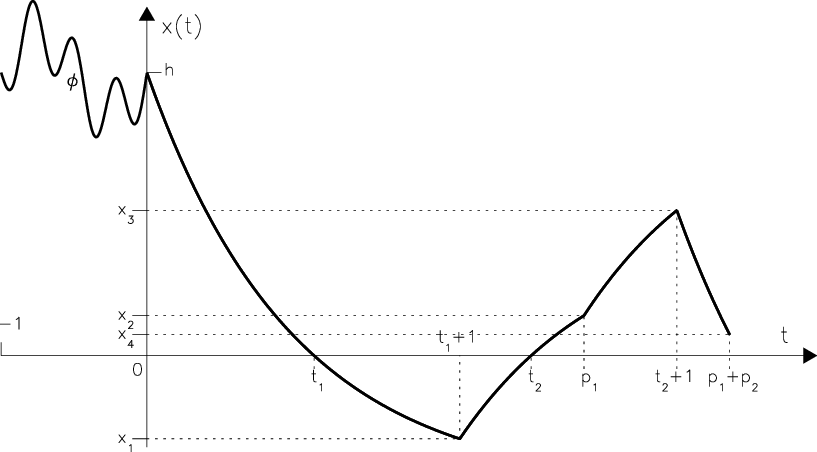}
    \caption{Solution of Equation (\ref{DDE-per}) with $x(p_1 + p_2) >0$.}
    \label{fig:single_period}
\end{figure}

Using the piecewise constant form of $f$ and $A$ with appropriate 
values on the interval $[0,p_1]$, one easily calculates the following results. 
$$
t_1=\frac{1}{\mu}\ln\frac{h\mu+a_1}{a_1},\quad x_1=x(t_1+1)=-\frac{a_1}{\mu}+\frac{a_1}{\mu}\exp\{-\mu\}<0.
$$
The second zero $t_2$ is calculated as 
$$
t_2=t_1+1+\frac{1}{\mu}\ln\frac{a_1-\mu x_1}{a_1}.
$$
The value $x_2$ is found from the condition $x_2:=x(p_1)$ as following
$$
x_2=x(p_1)=\frac{a_1}{\mu}+\left(x_1-\frac{a_1}{\mu}\right)\exp\{-\mu[p_1-(t_1+1)]\}>0.
$$
On the interval $[p_1,t_2+1]$ the value of $x_3$ at the endpoint is found as
$$
x_3=x(t_2+1)=\frac{a_2}{\mu}+\left(x_2-\frac{a_2}{\mu}\right)\exp\{-\mu(t_2+1-p_1)\}.
$$
Finally, the value of $x_4$ is defined as $x_4:=x(T)=x(p_1+p_2)$ and is found from integration on $[t_1+1, p_1+p_2]$ as
$$
x_4=x(p_1+p_2)=-\frac{a_2}{\mu}+\left(x_3+\frac{a_2}{\mu}\right)\exp\{-\mu[p_1+p_2-(t_2+1)]\}.
$$
In order for this solution to be slowly oscillating and of the desired shape we impose the following restrictions on the zeros $t_1, t_2$
and parametric values $a_1,a_2,p_1,p_2$: $0<t_1<t_1+1<t_2<p_1$ and $p_1-t_2<1$ (so that $p_1<t_2+1<p_1+p_2$) and $x_4>0$. These 
conditions are numerically  verified to be valid for  parameter values as given 
in Table~\ref{tab:soln_sp} below. 

\begin{table}[]
    \centering
    \begin{tabular}{|c|c|c|c|c|c|c|c|}
    \hline    
$~~~~~a_1~~~~~$ & $~~~~~a_2~~~~~$ & $~~~~~p_1~~~~~$ & $~~~~~p_2~~~~~$  & $~~~~~|m|~~~~~$ & $~~~~~b~~~~~$ 
& $~~~~~\mu~~~~~$ & $~~~~~T~~~~~$ \\
    \hline 
2	& 0.1	& 3	& 1	& 0.72 & 	1.76 & 0.1 &	4 \\    
2	&0.5	&3	&2.5&0.31 & 0.43 & 0.1&5.5	\\
3&	1&	3&	1&	0.21 & 1.69 & 0.1&	4\\	
3	& 0.1	& 3	& 7	& 0.41 &	1.03 & 0.1 &	10 \\
1	& 0.5	& 3	& 0.5	& 0.26 &	0.3 & 0.1 &	3.5 \\
4	& 1	& 3	&2.5	& 0.31&	0.87 & 0.1	& 5.5 \\
6	& 0.5	& 3	& 5	& 0.44 &	1.74 &	0.1 & 8 \\
6	& 2&	3&	1	&0.21	& 3.38 & 0.1 &	4 \\
1	 &0.17&	3	&0.5	&0.18 &	0.61 & 0.5 &	3.5 \\
1&	0.25&	3&	1&0.06 & 0.3 &	0.5 & 4	\\
2	& 0.1	& 3	& 0.5	& 0.33& 	1.43 & 0.5 &	3.5 \\
2	& 0.1	& 3	& 2.5	& 0.12 & 	0.4 & 0.5 &	5.5 \\
2	& 0.1	& 3	&3	& 0.1& 	0.27& 	0.1 & 6\\
2&	0.5&	3&	1& 0.06 & 0.6 &	0.5&	4\\
3&	1&	2&	1&	0.05 & 0.73 & 0.5&	3	\\
2&	1&	2&	1&	0.38 & 0.13 & 1&	3	\\
3&	0.25&	2&	1&	0.12 & 0.40 & 1 &3	\\
4	& 0.1	& 2	& 2.5	& 0.04 &	0.05 & 1 &	4.5\\
4	& 0.25	& 2	& 1	& 0.15 &	0.55&	1 & 3\\
4	& 3	& 2	& 1	&0.69 &	0.09 &	1 & 3\\
6	& 0.1	& 2& 	1	& 0.2 & 	0.87 & 1 &	3\\
6	& 0.1	&2	&3	&0.03 &	0.03& 1 &	5\\
    \hline
    \end{tabular}
    \caption{Examples of parameters, for which Equation~(\ref{DDE-per}) has a stable $T$-periodic solution.}
    \label{tab:soln_sp}
\end{table}

It is easily seen that the value of $x_4>0$ does not depend on the particular values of the initial function $\phi(s)$ on the interval 
$[-1,0)$; it is entirely determined by the positive value $\phi(0):=h>0$ only.

Introduce the one-dimensional map $F$ by $F: h\mapsto x_4.$ It is immediate that the fixed points $h_*$ of this map give rise to 
periodic solutions of equation (\ref{DDE-per}) with period $T=p_1+p_2.$
Namely, the initial state $\varphi$ of such a periodic solution is obtained 
by starting any solution $x$ with $x(0) = h_*$ and $x > 0$ on $[-1,0]$, and  then taking 
$\varphi := x_T =  x(p_1 + p_2 + \cdot )\resr{[-1,0]}
$.   
By using the above formulas for $x_1,x_2,x_3,x_4$ and $t_1, t_2$ and by an inductive calculation one arrives at the following expression 
for the map $F$:
\begin{equation}\label{F}
    x_4=F(h)=m\cdot h+b,
\end{equation}
where
\begin{equation}\label{m}
    m=\left[2\frac{a_2}{a_1}-\exp\{-\mu\}\right]\cdot\left(2-\exp\{-\mu\}\right) \exp\{-\mu(p_1+p_2-2)\},
\end{equation}
and 
\begin{align}\label{b}
    b=&\frac{2a_2-a_1\exp\{-\mu\}}{\mu}\left(2-\exp\{-\mu\}\right)\exp\{-\mu(p_1+p_2-2)\}\\
    &+\frac{a_1-a_2}{\mu}\exp\{-\mu p_2\}-\frac{a_2}{\mu}\nonumber
\end{align}

\begin{prop}\label{3.1}
    Suppose that the parameters $a_1, a_2, p_1, p_2$ are such that the inequalities $|m|<1$ and $b>0$ are satisfied. Then there exists a unique solution $h_*>0$ to the equation $F(h)=h$ which is given by the expression $h_*=b/(1-m)$, where $m$ and $b$ are defined by (\ref{m}) and (\ref{b}), respectively.
\end{prop}

\begin{cor}
    Under the assumptions of Proposition \ref{3.1} DDE (\ref{DDE-per}) has a stable slowly oscillating periodic solution $x: \R  \to \R$ of period $T$. 
    Its initial state   $\varphi = x_0 $ is obtained, as described above,  e.g. by starting with the  constant initial function $\psi \equiv h_*$ and applying the period map. The function $-x$ is then also a periodic solution, as follows from oddness of the right hand side of the equation. 
\end{cor}

%%
%% Subsection 3.2
%%
\subsection{Periodic Solutions with  period  twice the  period  of the coefficient $a$}
\label{subsec3.2}

In this subsection we are dealing with slowly oscillating  solutions of the form as shown in Fig.~\ref{tab:soln_dp}. 
That is, for given $h>0,$ the first zero $t_1$ of the solution satisfies $t_1<p_1, p_1<t_1+1<T,$ and $x(t)<0, \forall t\in(t_1,T]$ 
with $T-p_1>1$. 
This will be guaranteed if one assumes that $x_2<0$ and $x_3<0$. Using the piecewise constant form of $f$ and $A$ with appropriate 
values on the interval $[0,p_1]$, one easily derives the following calculations.
$$
t_1=\frac{1}{\mu}\ln\frac{h\mu+a_1}{a_1},\quad x_1=x(p_1)=-\frac{a_1}{\mu}+\left(h+\frac{a_1}{\mu}\right)\exp\{-\mu p_1\}.
$$
On the interval $[p_1, t_1+1]$ one finds:
$$
x_2=x(t_1+1)=-\frac{a_2}{\mu}+\left(x_1+\frac{a_2}{\mu}\right)\exp\{-\mu(t_1+1-p_1)\}<0.
$$
On the interval $[t_1+1,p_1+p_2]$ the value of $x_3$ is found as
$$
x_3=x(p_1+p_2)=\frac{a_2}{\mu}+\left(x_2-\frac{a_2}{\mu}\right)\exp\{-\mu[p_1+p_2-(t_1+1)]\}.
$$
The value of $x_3<0$ does not depend on the particular values of the initial function $\phi(s)$ on the interval $[-1,0)$; 
it is entirely determined by the positive value $\phi(0)=h>0$ only.

\begin{figure}
    \centering
    \includegraphics[width=0.95\textwidth]{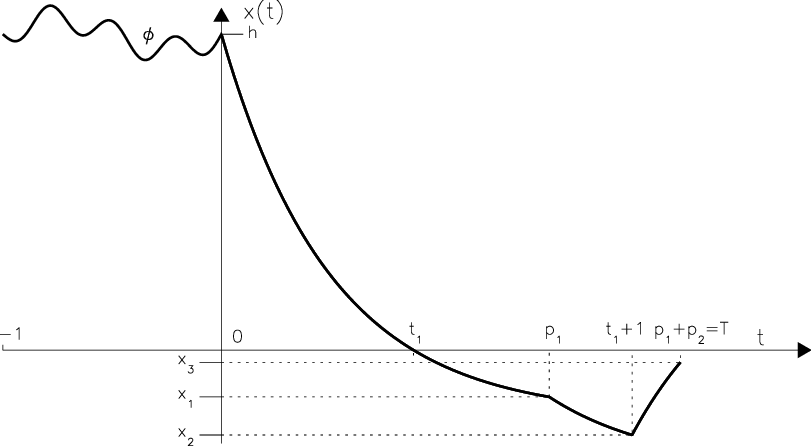}
    \caption{Solution of Equation (\ref{DDE-per}) with $x(p_1 + p_2) >0$.} 
    \label{fig:double_period}
\end{figure}

Introduce the one-dimensional map $\Phi_1$ by $\Phi_1: h\mapsto h_1:=x_3.$

Likewise, due to symmetry considerations, for every initial function $\psi$ satisfying $\psi(s)<0\; \forall s\in[-1,0]$ 
with $\psi(0):=h_1$ the corresponding solution $x=x(t,\psi), t\ge0,$ can be constructed with $x(p_1+p_2):=h_2>0$. 
This defines an analogous one-dimensional map $\Phi_2:h_1\mapsto h_2$ -- in fact, 
$\Phi_2(h) = -(\Phi_1 (-h))$

Fixed points $h_*$ of the composite one-dimensional map $\Phi=\Phi_2\circ\Phi_1$ give rise to slowly oscillating periodic solutions of
equation (\ref{DDE-per}). In particular, the positive solutions $h_*$ of the equation $\Phi_1(h)=-h, h>0,$ provide such fixed points of 
the map $\Phi$ (in general, there may exist additional fixed points of map $\Phi$ which are not solutions of this equation). 
We shall derive next some of the conditions on the parameters $a_1, a_2, p_1, p_2$ which guarantee the existence of the positive 
solutions to equation $\Phi_1(h)=-h$.

By using the recursive equations above for the values of $x_1, x_2$ and $x_3$ one finds that
\begin{equation}\label{x3}
    x_3=x(T)=x(p_1+p_2)=\Phi_1(h)=k\cdot h+d,
\end{equation}
where
\begin{equation}\label{k}
    k=\left(1-2\frac{a_2}{a_1}\exp\{\mu\}\right)\exp\{-\mu(p_1+p_2)\}
\end{equation}
and
\begin{equation}\label{d}
    d=\frac{a_2-a_1}{\mu}\exp\{-\mu p_2\}+\frac{a_1}{\mu}\exp\{-\mu (p_1+p_2)\}+\frac{a_2}{\mu}\left[1-2\exp\{-\mu (p_1+p_2-1)\}\right].
\end{equation}

\begin{prop}\label{3.2}
    Suppose that the parameters $a_1, a_2, p_1, p_2$ are such that the inequalities $|k|<1$ and $d<0$ are satisfied. Then there exists a unique solution $h_*>0$ to the equation $\Phi_1(h)+h=0$ which is given by the expression $h_*=-d/(m+1)$, where $k$ and $d$ are defined by (\ref{k}) and (\ref{d}), respectively.
\end{prop}

\begin{cor}
    Under the assumptions of Proposition \ref{3.2} DDE (\ref{DDE-per}) has a symmetric slowly periodic solution of period $2T$ satisfying
    $x(t+T)\equiv -x(t), \forall t\in\mathbb{R}$. As above, the solution can be obtained by starting with  the initial function $\varphi\equiv h_*$.
\end{cor}

\begin{table}[]
    \centering
    \begin{tabular}{|c|c|c|c|c|c|c|c|}
    \hline    
$~~~~~a_1~~~~~$ & $~~~~~a_2~~~~~$ & $~~~~~p_1~~~~~$ & $~~~~~p_2~~~~~$ & 
$~~~~~|k|~~~~~$ & $~~~~~d~~~~~$ & 
$~~~~~\mu~~~~~$ & $~~~~~T~~~~~$ \\
    \hline 
1&	0.17&1&	4&0.38 &-0.09&	0.1&	10\\
2&	0.5&	1&	2.5& 0.32 & -0.38 &0.1&	7\\
2&	0.25&	1&	5&	0.40 & -0.17 & 0.1&	12\\
4	& 2	& 0.5	& 1	& 0.09 & 	-1.72 & 0.1 & 	3 \\
4 &	0.1 &	1	& 7	&  0.43& 	 -1.39 & 0.1 &	16 \\
6 & 0.25 & 0.5 & 5 & 0.52 & -0.95 & 0.1 & 11 \\
2&	0.5&0.5&1& 0.08 & -0.49 &0.5&	3\\
2&	0.25&	1&	2.5&0.10&-0.09&	0.5&	7\\
2&	0.17&	1&	3&	0.1 & -0.09 &0.5&	8\\
3&	0.25&	0.5&2.5& 0.16& -0.11 &	0.5&	6\\
4 &	0.1	& 1	& 5	& 0.05	& -0.08 & 0.5 &	12 \\
4	& 1	& 0.5	& 1	& 0.08 &	-0.98 &	 0.5&	3 \\
2	&0.5	&0.5	&1 & 0.08 &	-0.21&	1	& 3	\\
2	& 0.1	 & 1	& 2.5	& 0.02 &	-0.01 &	1 &	7 \\
3	& 0.17	& 0.5	& 1	& 0.16 & 	-0.41 & 1 &	3 \\
4	& 1	& 0.5	& 1	& 0.08 &	-0.42& 1	&	3 \\
4	& 0.1	& 1& 	3	& 0.02 &	-0.03& 	1&	8 \\
6	& 1	& 0.5	& 1	& 0.02 &	-0.71 & 	1 & 	3 \\
    \hline
    \end{tabular}
    \caption{Samples of parameters, for which Equation~(\ref{DDE-per}) has stable periodic solutions with double period $2T$.}
    \label{tab:soln_dp}
\end{table}

%%
%%
%% Section 4: Smoothing Nonlinearities 
%%
%%
\section{Smoothing Nonlinearities}\label{Sec4}

We demonstrate how to get a smooth dynamical system close to the one generated by 
equation (\ref{DDE-per})  with the discontinuous coefficient function $a$ and $ f(x) = - \sign(x) $ 
from Section \ref{Sect3}. We focus on the case of the solutions with the same period as $ a$ here. 

 For given  functions $\tilde a, \tilde f: \R \to \R  $ and $\varphi \in C^0([-1,0], \R)$, 
we denote by $x^{\varphi}(\cdot , \tilde a, \tilde f)$ the solution to the equation 
$\dot x(t) = -\mu x(t) + \tilde a(t) \tilde f(x(t-1))$  with initial  state $x_0  =  x\resr{[-1,0]} =  \varphi$. 
As in previous sections, `solution' is to be understood in the sense that 
$x  = x^{\varphi}(\cdot , \tilde a, \tilde g)$ is continuous on an 
interval $[-1, t^+)$ and differentiable at all times $t \in [0, t^+) $  at which $a$ is continuous and where $x(t-1) \neq 0$,  and satisfies the equation for these 
$ t$. 
(At $t = 0$,  this refers to  right-side derivative). Since for our equations the forward solution depends only on 
$h = \varphi(0)$, we frequently also write $x^h(....)  $ instead of  $x^{\varphi}(....)$. 
Obviously such a solution satisfies  
\begin{equation} x(t) =  e^{-\mu(t-t_0) }\cdot x(t_0)  + \int_{t_0}^t e^{-\mu(t-s)} \tilde a(s) \tilde f(x(s-1)) \, ds   \quad (t \in [t_0, t^+)).  \label{varconst}   \end{equation}

\subsection{Preparations}
\label{subsec41} 

We  want to replace the step functions   $a$ and $ f = - \sign$ by nearby functions which are at least of class $C^1$,   preserving as much as possible of the original dynamics.
We need some technical preparation for  this: 

\begin{prop} \label{Prop41} Assume that  $\alpha, \beta  \in  \R$ with $\alpha \cdot \beta < 0$, that a function $ x$ satisfies $ x(t) = \alpha + \beta e^{-\mu t }$ on an open interval
$I $,  and $ x(t_0) = 0$ for some    $ t_0 \in I$.  Let  $\delta > 0$ be so small that $[-\delta, \delta] \subset x(I)$. 

\noindent Then, with  
$\displaystyle  \theta_+(\delta) :=   \frac{-1}{\mu}\ln( 1 - \frac{\delta}{\alpha}), \quad
 \theta_-(\delta) :=  - \frac{1}{\mu}\ln( 1 +\frac{\delta}{\alpha}), $ 
one has   $$  x(t_0 +\theta_-(\delta)) = -\delta,\;   x(t_0 + \theta_+(\delta)) = \delta, $$
and $ \theta_{\pm} (\delta) \to 0 $ as $ \delta \to 0$, locally uniformly w.r. to $ \alpha, \beta$  

\medskip \noindent b) For  a piecewise continuous function $\tilde f: \R\to \R$  one has  
$$ \int_{t_0 +\theta_-(\delta) +1}^{t_0 + \theta_+(\delta) +1} e^{-\mu[t_0 + \theta_+(\delta) +1 -s]}
 \tilde f(x(s-1)) \; ds =  J(\alpha,  \tilde f), \text{ with } $$  
$$ J(\alpha, g) :=  \frac{\alpha-\delta}{\mu} \int_{-\delta}^{\delta}  \frac{g(u)}{(\alpha - u)^2} \, du
\quad \text{  for }  g \in \mathcal{L}^1([-\delta, \delta), \R). $$ 
 (In particular, for fixed $ \delta$ and $ \mu$, the value of this integral depends  only on the limit value $\alpha$
of $\alpha + \beta e^{-\mu t } $ for  $ t \to \infty$, but not on $ \beta$.) 
\end{prop} 
\begin{proof} 
(For the proof,  we omit the $\delta-$dependence of $ \theta_{\pm} $ in the  notation.) 

Ad a): $\theta_+ $ is characterized by the equation $\beta\cdot [e^{-\mu(t_0 + \theta_+)} - e^{-\mu t_0}] = \delta$, and 
$\beta  e^{-\mu t_0}  = - \alpha$ gives:  $  -\alpha [e^{-\mu \theta_+} - 1]  = \delta $, or 
$$\theta_+ =  \frac{-1}{\mu}\ln( 1 - \frac{\delta}{\alpha}).$$ The second formula is obtained analogously, and the convergence result is obvious from these formulas.

\medskip Ad b): Note that  $e^{\mu s} = \beta/(x(s) - \alpha) \;(s \in I)$, in particular, 
$$e^{-\mu (t_0 + \theta_+)} = (x( t_0 + \theta_+) -\alpha)/\beta = (\delta - \alpha)/\beta.$$
Note also that $\dot x(s) = -\mu(x(s) - \alpha)\;  (s \in I)$. 
  Using  these identities  and the substitution rule,   one sees that the integral in question equals 
$$ \begin{aligned} & \frac{\delta - \alpha}{\beta}  \int_{t_0 + \theta_-}^{t_0 + \theta_+} e^{\mu s}  \tilde f(x(s))\, ds    =  \frac{(\delta - \alpha)}{\beta} \int_{t_0 + \theta_-}^{t_0 + \theta_+} \frac{\beta}{x(s) - \alpha} 
\tilde f(x(s)) \cdot \frac{\dot x(s)}{-\mu (x(s) -\alpha)  } \, ds \\
&=   \frac{ (\delta - \alpha)}{\beta} \int_{-\delta}^{\delta}   
 \frac{\beta}{-\mu (u - \alpha)^2} \tilde f(u) \, du = 
\frac{\alpha-\delta}{\mu} \int_{-\delta}^{\delta}  \frac{\tilde f (u)}{(\alpha  - u)^2} \, du = J(\alpha, \tilde f).
\end{aligned} 
 $$
\end{proof} 

The main idea of our smoothing procedure is not at all new, but present in previous publications; 
  see \cite{AdHW}, \cite{AdHM}, \cite{Pet83}, \cite{Wal81b},  and \cite{HaleLin} for papers using step functions  and (except for \cite{AdHM} and \cite{Pet83}) smoothed versions of these as nonlinearities of delay equations. 
The most important point in these considerations is  mentioned in \cite{HaleLin} on the top of p. 696: The values of the solution $x$ at times  $t_0$ where its   derivative has jump discontinuities 
(which arise as a consequence of the discontinuities of $f$ or $a$ in our situation) have to be a minimum distance away from the discontinuity points of points  of  the step function $f$  (in our case, only zero), so that 
$f$  and also the smooth version $ \tilde f$  of $f$  are constant in a neighborhood of these values. 
Then   the  (constant) values  of $f(x(s)) $ and of $ \tilde f(\tilde x(s))$, where the  tilde indicates smoothed nonlinearity and corresponding solution,  coincide for $s$ in a neighborhood of the time  $ t_0$ where $\dot x $ has a jump  discontinuity.

\subsection{Smoothing $f = - \sign$ to $\tilde f$}
\label{subsec42} 
Consider solutions  $ x = x^h$ starting with $x(0) = h $ as  described in Section \ref{subsec3.1}, with zeroes at 
$t_1$ and $t_2 $ (depending on $h$), for $ h$-values close to the value $  h_*$ 
corresponding to the periodic solution as constructed in Section \ref{subsec3.1}. 
For small enough $\delta >0$, there exist the numbers 
 $ \theta_{\pm} = \theta_{\pm}(\delta) $ associated to the first crossing of zero at $t_1$ as in part a) of Proposition \ref{Prop41}, and analogously $\vartheta_{\pm} $  associated to the second zero at $t_2$. 
(Note that then $\theta_- > 0 > \theta_+$ and $\vartheta_- < 0 < \vartheta_+$.)

\begin{prop} \label{Prop42} For $\delta > 0 $ sufficiently small, there exists a $C^{\infty} $ function 
$\tilde f: \R \to \R$ with 
\begin{equation} \tilde f(x) = f(x) = -\text{\rm sign}(x) \text{ if } |x| \geq \delta, \text{ and }  -1 \leq \tilde f(x) \leq 1 \text{ if }  |x| \leq \delta,  \label{smoothf} \end{equation} 
and such that the solutions 
$\tilde x =  x^{\varphi}(\cdot,a, \tilde f) $ and 
$x =  x^{\varphi}(\cdot,a, \tilde f) $ (both still with the discontinuous coefficient $a$)  
 with $ \varphi(0) = h$, where $h$ is close to $h_*$,   have the following properties: 

\medskip \noindent a) Let  $ \theta_{\pm} $ and $\vartheta_{\pm} $ be as explained above. Then
(with the $h$-dependent times $t_1$ and $t_2$), 
 $$\tilde x = x   \quad \text{ on  }   [-1, t_2 + 1+ \vartheta_-] \setminus [t_1+ 1 + \theta_+, \; t_1+ 1+\theta_-].$$

\medskip \noindent b) There exists a constant $ C >0$ such that on the`exceptional'interval\\  $[t_1+ 1+\theta_+,\;  t_1+ 1+\theta_-]$ one has   $|\tilde x(t) - x(t) | \leq C\delta,  $
and also 

\medskip \noindent c)  $  |\tilde x - x|   \leq C\delta$ on $  [t_2 + 1 + \vartheta_-, t_2 + 1 + \vartheta_+]$.  

\medskip \noindent d) On $[t_2 + 1 + \vartheta_+, p_1+p_2]$ one has  for  sufficiently small values of $\delta $ 
and  of $ |h - h_*|$: 
$$\tilde x(t) - x(t)  = 
 e^{-\mu[t-(t_2 + \vartheta_+ + 1)]}R(\delta),  $$
where $R(\delta) :=  a_2 \cdot  J(a_1/\mu, \tilde f- f)$ is independent of $h$, and $|R(\delta)| \leq C\delta. $ 
\end{prop} 
\begin{proof}  Clearly there exist $ C^{\infty}$ functions $\tilde f$ with the properties in (\ref{smoothf}). 
For  such a function and small enough $\delta >0 $ (with $\delta < x_1$, see Section \ref{subsec3.1}), 
one has $x = \tilde x$ on $[-1, t_1 + \theta_+ +1]$. We want to achieve 
 $\tilde x(t_1 + \theta_- +  1)  = x(t_1 + \theta_ - + 1)$. In view of  (\ref{varconst}), applied with 
$t_0 = t_1 +  \theta_+  +1$, and the fact that $a= a_1$ is constant around $t_1 + 1$, 
 this condition  is equivalent   to
\begin {equation} \int_{t_1 +  \theta_+ +1}^{t_1 + \theta_ - +1} e^{-\mu[t_1 + \theta_- +1 -s]}
 \tilde f(\tilde x(s-1)) \; ds  = \int_{t_1 + \theta_+ +1}^{t_1 + \theta_- +1} e^{-\mu[t_1 + \theta_- +1 -s]}
  f(x(s-1)) \; ds. \label{cond1}\end{equation}
Now on the interval $[t_1 + \theta_+,  t_1 + \theta_-] $ (where $t_1$ depends on $h$),  solutions $x = x^h$  have the form $t \mapsto \alpha + \beta e^{-\mu t}$, where $ \alpha = -a_1/\mu$,  
and only $$  \beta = (a_1/\mu) e^{\mu t_1} = \frac{a_1}{\mu}\cdot \frac{\mu h + a_1}{a_1} = h + a_1/\mu $$   depends on $h$, see Section \ref{subsec3.1}.
The formulas for $\theta_{\pm} $  from part a) of Proposition \ref{Prop41} thus show that 
$\theta_{\pm}$ are independent of $ h$. Further, part b) of the same Proposition shows that condition (\ref{cond1}) 
is equivalent to 
\begin{equation} J(\alpha, \tilde f) = J(\alpha, -\sign).  
  \label{cond2}\end{equation}
The number $ c = J(\alpha, -\sign) $  on the right hand side of equation (\ref{cond2}) is clearly between the numbers 
$J(\mathbf{\alpha,-1})$ and $J(\alpha,\mathbf{+1})$, where $\mathbf{\pm 1}$ denotes the constant functions.  
By choosing $ \tilde f_-, \tilde f_+ $ according to (\ref{smoothf})  close  enough to $\mathbf{-1}$ and 
 $\mathbf{+1}$ in $ \mathcal{L}^1([-\delta, \delta])$,  respectively,  one  can achieve that $ c$ is also between  
$J(\alpha, \tilde f_-)$ and $ J(\alpha,\tilde  f_+)$. It follows then  by the intermediate value theorem 
(using  continuity of $J(\alpha,\cdot) $ with respect to the $\mathcal{L}^1$-norm)   
that  there exists  $s \in [0, 1]$ such that the   function $\tilde f  = (1-s) \tilde f_- + s \tilde f_+$  satisfies condition (\ref{cond2}).

With such an $\tilde f$, one has on the interval  $[t_1+ 1+\theta_+,\;  t_1+ 1+\theta_-]$ that 
\begin{equation} \begin{aligned} |\tilde x(t) - x(t) |  &= |   \int_{t_1 +  \theta_+ +1}^{t} 
\underbrace{e^{-\mu[t -s]}}_{\leq 1} a_1 \cdot \underbrace{[\tilde f(x(s-1)) -  f(x(s-1))]}_{|...| \leq 2}  \, ds | \\ 
& \leq (\theta_- -\theta_+)\cdot 2 |a_1| \leq C_1 \delta  \end{aligned}  \label{C1} \end{equation} 
with appropriate $C_1$, where the last estimate follows from the formulas for $ \theta_{\pm}$  from Proposition \ref{Prop41}.
If $ \delta $ is so small that $x_1 + C_1 \delta < -\delta$ then for all solutions 
$\tilde x$ in question 
$\tilde f(\tilde x(s)) = +1 = f(x(s)) \;  (s \in  [t_1+ 1+ \theta_+, \; t_1 + 1+\theta_-]),$ which together with 
$ \tilde x (t_1+1 + \theta_-) =   x (t_1+1 + \theta_-) $   gives that also 
$x = \tilde x$ on $[t_1+1 + \theta_-, t_1 + 2+ \theta_-]$. This interval may or may not contain the second zero 
$t_2$  of $x$, but in any case we have $ a_1 \tilde f (x(s)) = a_1 f(x(s))   = +a_1 $ for 
$s \in  [t_1+1 + \theta_-, t_2 +\vartheta_-]$. Thus we obtain
$$ x = \tilde x \text{ on } [t_1+1 + \theta_-, t_2 + 1 +\vartheta_-],  $$
and a) is proved. 
 
Clearly, since $t_2 <  p_1 < t_2 + 1$ and  $ \vartheta_{\pm} \to 0 $ as $ \delta \to 0$, we can assume that 
$t_2 + \vartheta_+ < p_1 < t_2 + 1-\vartheta_-$.  
On the interval $[t_1 + 1+ \theta_-, p_1]$, which contains $t_2$ in its interior, and on which $ x = \tilde x$,
 we then have  $x(s) = \alpha_2 + \beta_2e^{-\mu s} = \tilde x(s) $,  with $\alpha_2 = a_1/\mu$  again independent of $h$ and (see Section  \ref{subsec3.1}, and observe $x(t_1 + 1) = x_1$)
$$\begin{aligned} \beta_2 & = (x_1 - a_1/\mu) e^{\mu(t_1 + 1)} = (a_1/\mu)(-2 + e^{-\mu})e^{\mu} \frac{\mu h + a_1}{a_1}  \\ 
& = (1 - 2e^{\mu})\cdot  \frac{\mu h + a_1}{\mu} = (1 - 2e^{\mu})(h + a_1/\mu). 
\end{aligned} $$ It follows as  above for $ \theta_{\pm}$ that  the numbers $\vartheta_{\pm} = \vartheta_{\pm}(\delta)$    are independent of  $h$, and given by the formulas from Proposition \ref{Prop41} with  $\alpha : = \alpha_2 = a_1/\mu$. On the interval  $[t_2 + \vartheta_-, t_2 + \vartheta_+]$, the solutions 
$x$ and $ \tilde x$ coincide and pass through $[-\delta, \delta]$ a second time, but (obviously) in a different way.   
Since $ \tilde f$ is already chosen,   the analogue of condition (\ref{cond2}) with $\alpha_2$  instead of 
$\alpha$ will not hold, so that we do  not obtain $\tilde x(t_2 + \vartheta_+ + 1) = x(t_2 + \vartheta_+ + 1)$.
However, we have $\tilde x(t_2 + \vartheta_- + 1) = x(t_2 + \vartheta_- + 1)$, and  for 
$t \in [t_2 + \vartheta_- +  1, t_2 + \vartheta_+ + 1]$ we have in analogy to  (\ref{C1}) 
$$  \tilde x(t) - x(t)   =   \int_{t_2 +  \vartheta_- +1}^{t}  e^{-\mu(t-s)} a_2 \cdot [\tilde f(x(s-1)) -  f(x(s-1))]   \, ds $$
(recall that $ \tilde x = x$ on $[t_2 + \vartheta_-, t_2 + \vartheta_+]$),   
which gives the  estimate 
$| \tilde x(t) - x(t)|  \leq (\vartheta_+ - \vartheta_-)\cdot 2a_2  \leq C_2\delta,$
with appropriate $C_2$. Assertions b) and c) follow with $C := \max\{C_1, C_2\}$. \\
Further, for  $t =  t_2 + \vartheta_+ + 1$ in particular, we obtain from Proposition \ref{Prop41} b): 
\begin{equation} \begin{aligned}
 & \quad \tilde x( t_2 + \vartheta_+ + 1) - x( t_2 + \vartheta_+ + 1) \\
&   =   \int_{t_2 +  \theta_- +1}^{ t_2 + \vartheta_+ + 1}  e^{-\mu( t_2 + \vartheta_+ + 1 -s)} a_2 \cdot [\tilde f(x(s-1)) -  f(x(s-1))]   \, ds \\
& = a_2 \cdot  J(\alpha_2, \tilde f- f), \quad \text{ where } \alpha_2 = a_1/\mu. 
\end{aligned} \label{Rdelta}  \end{equation}
Obviously the last quantity, which we abbreviate by $R(\delta)$,  is independent of  the `initial value' $h$
and  satisfies  $|R(\delta)| \leq C_2\delta  \leq C\delta. $

Ad d):  From (\ref{varconst}) and (\ref{Rdelta})  we see that  for $t \in   [t_2 + \vartheta_+ + 1, p_1+p_2]$
$$\tilde x(t) - x(t) = e^{-\mu[t-(t_2 + \vartheta_+ + 1)]}R(\delta)  +  
 \int_{t_2 +  \vartheta_+ +1}^{t}  e^{-\mu(t-s)} a_2 \cdot [\tilde f(\tilde x(s-1)) -  f(x(s-1))]   \, ds,
$$
Now for $h$ close to $h_*$ and $\delta$ sufficiently small, both  terms $\tilde f(\tilde x(s-1)) $ and 
$ f(x(s-1))$ in this integral equal $-1$,  since the periodic solution $x^{h_*}$  decreases from its maximal value at 
$t_2(h_*) + 1 $ to  the still positive value $h_* $ on $[t_2(h_*), p_1 + p_2]$, and the difference  
$\tilde x(t) - x(t) $ equals the first summand in the above expression as long as 
also $\tilde x(s-1) \geq \delta$,  which is true if  $ h_* + R(\delta)  >  \delta $ and $|h-h_*|$ is sufficiently small. We thus obtain the assertion of d).
\end{proof} 

\subsection{Smoothing the coefficient function  $a$  to $ \tilde a$}\label{subsec43}

\begin{prop} \label{Prop43}
For sufficiently small $ \rho >0$,  the step function $a$  with the two positive values $ a_1, a_2$ can be replaced by  a  function $\tilde  a:\R \to \R $  which is of class $C^{\infty} $, also has period $p_1 + p_2$,  and  the following properties: 

\medskip  \noindent a) $\tilde  a = a = a_1 \text{ on } [\rho, p_1], \quad 
\tilde  a = a = a_2 \text{ on } [p_1+ \rho, p_1 + p_2], $ hence also $\tilde  a = a = a_2$ on $[-p_2 + \rho, 0]$. 

\medskip  \noindent b) $  0 <  \tilde a \leq 2 \max\{a_1, a_2\} $ on the `transition intervals' $[0, \rho]$ and $[p_1, p_1 + \rho]$. 

\medskip  \noindent c) For $ h $ close to $ h_*$, the solutions $\tilde{\tilde x} = x^h(\cdot, \tilde a, \tilde f)$ and $\tilde x = x^h( \cdot, a, \tilde f) $ as in  Proposition \ref{Prop42}  satisfy   
$$ \tilde{\tilde x} = \tilde x \quad \text{ on } [0, p_1 + p_2]\setminus([0, \rho] \cup [p_1, p_1 + \rho]),$$ 
 and  on the `exceptional' intervals  $[0, \rho]$ and $[p_1, p_1 + \rho]$ one has 
$$|\tilde{\tilde x} - \tilde x| \leq 2 \max\{a_1, a_2\} \rho.$$  
\end{prop} 
\begin{proof} We first replace $ a $  on   an interval  $[0, \rho]$   with $ \rho > 0 $ small
by a $C^{\infty}$ function $ \hat a$ with $\hat a (0) = a_2, \hat a(\rho) = a_1$ (taking $  \rho < t_1(h_*)$).    
The condition for the corresponding solution $\hat x = x^h(\cdot, \hat a, \tilde f)$ to satisfy 
$ \hat x(\rho) = \tilde x(\rho)$ is 
$$ \begin{aligned} \int_0^\rho \hat a(s) e^{-\mu(\rho-s)}\cdot (-1)  \, ds  & = \int_0^\rho  a(s)  e^{-\mu(\rho-s)}\cdot (-1)  \, ds, \text{ or } \\
 \int_0^\rho \hat a(s) e^{-\mu(\rho-s)} \, ds  & = a_1 \int_0^\rho  e^{-\mu(\rho-s)} \, ds. \end{aligned}$$
This equality can be satisfied,  along with the boundary conditions,   by a function 
$\hat a$ taking values in $ (0, a_2]$ if $ a_1 < a_2$, and with values in $(0, 2a_1]$ if $a_2 < a_1$, so in any case  
with values in $(0, 2\max\{a_1, a_2\}]$. With such a choice, we obtain for $ t \in [0, \rho]$ 
$$|\hat x(t) - x(t) | \leq  2\max\{a_1, a_2\}\int_0^t e^{-\mu(t-s} \, ds \leq 2\max\{a_1, a_2\}\rho, $$
and $ \hat x(\rho) = \tilde x(\rho) $. For small enough $\rho$  we conclude that 
$\tilde f(\hat x(t)) = \tilde f(\tilde x(t)) = -1 $ for $t \in [0, \rho] $, implying $\hat x = \tilde x $ 
on $[\rho, p_1 + p_2]$, where $ \hat a = a$.

In a similar second step we modify $\hat a$ to $\tilde a$ on the interval $[p_1, p_1 + \rho]$ 
(choosing $ \rho >0$ such that $p_1 + \rho < t_2(h_*) + 1$) to a  $C^{\infty}$ function $  \tilde a$ 
with period $p_1 + p_2$,  with $\tilde a(p_1) = a_1, \tilde a(p_1 + \rho) = a_2$, and $\tilde a = \hat a$  
  on $ [0, p_1 + p_2]\setminus [p_1, p_1 + \rho]$. Again this can be done with 
$\tilde a $ taking values in $(0,  2\max\{a_1, a_2\}]$ and such that the solutions 
$\hat x  = x^h(\cdot, \hat a, \tilde f) $  (as above)  and $\tilde{\tilde x} := x^h(\cdot, \tilde a, \tilde f) $
with  the same value $h$ close to $ h_*$ at  time zero  satisfy 
$\hat x(p_1 + \rho) = \tilde{\tilde x}(p_1 + \rho)$. As above, then 
for  $\rho$  small enough (so that   $\tilde{\tilde x} \geq \delta$ on $ [p_1, p_1 + \delta]$), this gives $\tilde{\tilde x} = \hat x$ on 
$[p_1 + \rho, p_1 + p_2]$.

\end{proof} 

\begin{cor}\label{Cor44} 
For the equation $\dot x(t) = - \mu x(t) + \tilde a(t) \tilde f(x(t-1))$, with 
$\tilde a$ and $ \tilde f$ as  above,  and  for $ h$ close to $  h_*$, the map 
$\tilde F:   h = x(0) \mapsto x(p_1 + p_2) = \tilde F(h) $ is given 
by $$\tilde F(h) = F(h) + \eta(h,  \delta)R(\delta)  = mh + b + \eta(h, \delta)R(\delta), \quad \text{ with }  $$
$$\eta(h, \delta) := \exp[\mu( 1 - (p_1 + p_2))](2e^{\mu} -1) \frac{a_1 +\mu h }{a_1 - \mu \delta },$$   
with $R$ from Proposition \ref{Prop42}, and $M, b$ as in Section \ref{subsec3.1}.
\end{cor} 
\begin{proof} 
The formulas for $t_1, x_1$ and $t_2 $ from Section  \ref{subsec3.1} give 
$x_1 = \frac{a_1}{\mu}(e^{-\mu} -1)$ and
$$ \begin{aligned} t_2 = t_2(h)  & = t_1(h) + 1 + \frac{1}{\mu}\ln[\frac{a_1 - \mu \frac{a_1}{\mu}(e^{-\mu}-1)}{a_1}] = 
t_1(h)  + 1  + \frac{1}{\mu}\ln(2-e^{-\mu}) \\
& = 1 + \frac{1}{\mu}[\ln(\frac{a_1 + \mu h }{a_1 })  + \ln(2-e^{-\mu}) ] \\
& = 1 +  \frac{1}{\mu}\ln[(2-e^{-\mu}) \frac{a_1 + \mu h}{a_1}]. \end{aligned} $$ 
Thus $\displaystyle e^{\mu t_2(h)}  = e^{\mu}(2-e^{-\mu})\frac{a_1 + \mu h}{a_1} = (2e^{\mu} -1)\frac{a_1 + \mu h}{a_1}.$ 
Further note that from the construction of $\vartheta_+(\delta)$ and Proposition \ref{Prop41}, 
$$  e^{\mu\vartheta_+(\delta)}   = \frac{1}{1 - \delta/\alpha_2}  = \frac{1}{1 - \mu\delta/a_1} = 
  \frac{a_1 }{a_1 - \mu \delta }.$$ 
 Combining  part c) of Proposition \ref{Prop43}
and part d) of  Proposition \ref{Prop42} we see that 
$$\begin{aligned} \tilde F(h) & = \tilde{\tilde x}(p_1 + p_2) = x(p_1 + p_2) + \exp\{-\mu[p_1 + p_2-(t_2(h)+ \vartheta_+(\delta) +1)]\}  R(\delta) \\
& =  F(h) + \exp[\mu( 1 - (p_1 + p_2))]  e^{\mu t_2(h)}  e^{\mu\vartheta_+(\delta)} 
  R(\delta) \\
& =   F(h) +  \exp[\mu( 1 - (p_1 + p_2))](2e^{\mu} -1)\frac{a_1 + \mu h}{a_1} \frac{a_1 }{a_1 - \mu \delta }
 R(\delta) \\
& = F(h) + \exp[\mu( 1 - (p_1 + p_2))](2e^{\mu} -1) \frac{a_1 +\mu h }{a_1 - \mu \delta }R(\delta) \\
& = F(h) + \eta(h,\delta) R(\delta). 
\end{aligned} $$ 
\end{proof} 

With the smoothing as in Corollary \ref{Cor44}, we  obtain the following theorem for the equation 
$$ \dot x(t) = - \mu x(t) + \tilde a(t) \tilde f(x(t-1)) \leqno{(\tilde a, \tilde f)}  $$
\begin{thm}\label{Thm45}  For a smoothed equation $ (\tilde a, \tilde f)$ as in Corollary \ref{Cor44} and $ \delta >0$  small enough,  the map  $\tilde F  $   a has a unique fixed point $\tilde h$ close to $h_*$, with 
$$ \tilde h \to h_*  \; \text{ and }\;  \tilde F'(\tilde h) \to 
 F'(h_*) = m  \; \text{ as } \delta \to 0. $$ To this fixed point corresponds a periodic solution $\tilde p: \R \to \R  $  of 
equation $ (\tilde a, \tilde f)$ with (minimal) period  $p_1 +p_2$. 
\end{thm} 
\begin{proof}  The proof follows from the construction and from the fact that in a neighborhood of $ h_*$, 
the map $ \tilde F$ converges to the affine linear map $F$ in the $C^1$-norm, as one sees from the 
description of their  difference in  Corollary \ref{Cor44} and the property 
$ R(h) \to 0\;  (\delta \to 0) $ from Proposition \ref{Prop42}. 
 \end{proof}

\subsection{ Computation of spectral properties of the smoothed period   map} 
\label{subsec44}

 Consider  a smoothed equation $(\tilde a, \tilde f) $ as in Theorem \ref{Thm45} with $ \delta $  small enough so that 
the corresponding periodic solution $  \tilde p $ exists.  
(We do generally not denote the dependence  of  objects on $ \delta$ or other details of the smoothing.)
Note  that the initial state  $\tilde\varphi:= 
\tilde p_0$ of  this periodic solution  is given by $p_0 = \tilde {\tilde x}_{p_1 + p_2}$, where as above, 
$\tilde {\tilde x} = x^{\tilde h}(\cdot, \tilde a, \tilde f) $, and one can take the constant function equal to 
$\tilde h$   as initial  state for $\tilde {\tilde x}$. 
We briefly write $C^0$  for  the space $C^0([-1, 0], \R)$ with the norm given by\\ 
$$||\varphi||_{\infty} = \max\menge{|\varphi(t)|}{t \in [-1, 0]},$$
and define the  evaluation map  $ \text{ev}_0: \varphi \in C^0([-1, 0], \R) \to \R, \; \text{ev}_0(\varphi) := \varphi(0)$.
Set  $$D^+ := \menge{\varphi \in C^0}{ \varphi([-1, 0]) \subset [\delta, \infty) }.  $$  

For initial functions  $\varphi \in D^+$,  (and with the `starting time' $t_0 = 0$)
the values of the  corresponding solution $x^{\varphi}$ on $[0, 1]$  (the first interval of delay length) depend  only  on $h :=  \varphi(0) = \text{ev}_0(\varphi) $, and hence the same is true for the solution on 
all of $[0, \infty)$.  From Theorem \ref{Thm45}, the fixed point $\tilde h > 0$ of $\tilde F$ (close to $h_*$)  
exists. 

Set $ T := p_1 + p_2$, and  define the period  map $ P(T,0): C^0 \ni \varphi \mapsto x^{\varphi}(\cdot, \tilde a, \tilde f)_T \in C^0$, which is the forward evolution operator  for  the nonautonomous equation from  starting time $0$ to the period  $T$.  The fact that $\tilde a$ and $ \tilde f$ are of class $C^1$ (for $ \tilde a$, continuity would be sufficient) implies that $ P$ is of class $C^1$ w.r. to the norm $||\; ||_{\infty}$ 
(compare, e.g., Theorem 4.1, p. 46 in \cite{Hale77}). 
 Due to the nonautonomous character of the equation, the starting time  is also important here,
 not only the time difference $T$.    Analogous time evolution operators $ P(t, s) $ can be defined for all $ t, s \in \R$, which then have the cocycle property $P(t,s)\circ  P(s, u) = P(t, u) $ (if $  t \geq s \geq u$).  $T$-periodicity of the  equation  implies that
\begin{equation} P(nT, 0) = [P(T,0)] ^n \; (n \in \N_0). \label{Piterate}\end{equation} 
 We   mostly write $P$ for  $P(T,0)$ from now on. It follows from \myref{Piterate} that 
fixed points of $P$ give periodic solutions of the equation, and that iteration of $P$
describes the solution behavior also for non-periodic  solutions. In particular, 
$$ P( \tilde \varphi) = \tilde \varphi. $$
For an appropriate open interval $I\subset \R $   containing $\tilde h$, and the set 
$\tilde D^+ := \menge{\varphi \in D}{ \varphi(0) \in I}$ 
 we have a commutative diagram 
$$\resizebox{4.5cm}{!}{
\xymatrix{
\tilde D^+ \ar[r]^{P}  \ar[d]_{\text{ev}_0} & D^+ \ar[d]^{\text{ev}_0}  \\
I \ar[r]^{ \tilde F} & \R   
}
}
$$
Now, for  $ h \in I$,  consider the constant  functions $h\cdot \mathbf{1}  \in C^0$, where  $ \mathbf{1} \in C^0 $ is the constant function with value 1.  Define 
$\Gamma: I \ni h \mapsto  P (h\cdot \mathbf{1}) \in C$; then $\Gamma: I   \to C^0$ is a   $C^1$-curve
(we do not denote the dependence of $ \Gamma$ on $ \delta $.) 
The fact that forward solutions  depend only on the value of the initial  segment at $t = 0$   
implies that   for a whole neighborhood  $U$ in $C^0$ of the initial  segment $\tilde \varphi$ of the periodic solution (the fixed point of $P$), one has 
\begin{equation} P(U) \subset \text{image}(\Gamma), \label{intoGamma} \end{equation}
and we have
$$\Gamma(\tilde h ) =   P [ \tilde h \cdot \mathbf{1}] =   P(\tilde \varphi)   = \tilde \varphi.$$
We know  (assuming $ \delta $ small enough) that $ \lambda :=  \tilde F'(\tilde h) \approx m$   is not zero. 
Note that the map $I\ni h \mapsto \text{ev}_0(h\cdot  \mathbf{1})$ is the identity on $I$, and hence we have 
$$ \tilde F'(\tilde h) = \frac{d}{dh}\resr{\tilde h} \{ \tilde F\circ ev_0 \circ [h \mapsto h\cdot  \mathbf{1}] \}.$$ 

Using the  the above diagram  we obtain 
\begin{equation}  \begin{aligned} 0 \neq \lambda = \tilde F'(\tilde h)  &= \frac{d}{dh}\resr{h = \tilde h}\{\tilde F \circ \text{ev}_0 \circ [h \mapsto h\cdot \mathbf{1}]\}\\
&= \frac{d}{dh}\resr{\tilde h}  \left\{ \text{ev}_0 \circ P \circ [h \mapsto h\cdot \mathbf{1}]\right\} \\
& = \frac{\rm d}{\rm{d}h}\resr{h = \tilde h}  [\text{ev}_0\circ \Gamma]  = \text{ev}_0[\Gamma'(\tilde h)].
\end{aligned} \label{ev0Gamma'} \end{equation}
 Thus, the tangent vector   $  v := \Gamma'(\tilde h) \in C^0 $ is nonzero, and even  $v(0)\neq 0$. 
 It follows that   there exist an open  interval $I^*$   around $\tilde h$   such that 
 $W^* := \text{image}(\Gamma\resr{I^*})  $ is a one-dimensional $C^1$-submanifold of $C^0$, containing  $\tilde \varphi$.  
A neighborhood of $\tilde \varphi$ in $C^0$  is mapped into  $W^*$ by $P$. Therefore, 
 with the one dimensional tangent  space $  \mathcal{T} :=  T_{\tilde \varphi}W^* = \R\cdot  \Gamma'(\tilde\varphi) 
 = \R\cdot v$ and the  derivative $ M := DP(\tilde\varphi) $  one has 
$\text{image} (M) \subset  \mathcal{T}$.
We conclude that  the kernel $ \mathcal{N}$ of $M$ is a closed subspace of $ C^0$ with codimension at most one,  and $M \mathcal{T} \subset \mathcal{T}$, so that $ v$  is an eigenvector of $M$. 
(We show next that the corresponding eigenvalue equals $ \lambda$, which comes as no surprise.)  
,,Differentiating'' the above diagram at $  \tilde \varphi$ into the direction of
$v$,  we obtain 
$$\begin{aligned} (Mv)(0)  & = \text{ev}_0(M(v)) = 
\text{ev}_0(DP(\tilde\varphi) (v)) \\
&= \tilde F'(\text{ev}_0( \tilde \varphi))\cdot \text{ev}_0(v)  = 
 \tilde F'(\tilde h) \cdot \text{ev}_0(v) = \lambda \cdot v(0).
 \end{aligned} $$
  Since $v(0) \neq 0   $, and we know  already that $v$ is an eigenvector of $M$, the associated eigenvalue equals $ \lambda$.  Summing up, we  obtain the following  result: 

\begin{thm} \label{Thm46}
Assume that for the original discontinuous equation the map $F$ defined by $  F(h) = mh + b$  as in Section \ref{subsec3.1} has a fixed point $h_* >0$. Then, for  smoothed  equations $(\tilde a, \tilde f) $ as above, 
if $ \delta > 0 $ is small enough, there exist a unique periodic solution $\tilde p$ with period $  T = p_1 + p_2$ 
initial state $P_0 = \tilde \varphi$, where $\varphi(0) = \tilde h$ is the unique fixed point of the associated one-dimensional  
map $\tilde F$. The corresponding period map $P = P(T,0) $ with $P(\tilde \varphi) = \tilde \varphi$ maps a 
neighborhood of $ \tilde \varphi$ in  $  C^0$ into a one-dimensional submanifold of $  C^0 $. 
For its derivative $M =DP(\tilde\varphi)$, we have the   $M-$invariant decomposition 
$$C^0 = \mathcal{N} \oplus \mathcal{T}$$ 
corresponding to the eigenvalues  $ 0$  and $  \lambda$ of $M$. 
Further, $ \lambda \to m$ as $ \delta \to 0$, so, in particular,  the periodic  solution 
$\tilde p$ is stable  for equation $(\tilde a, \tilde f) $ if $ m < 1$,  and $\delta $ is small enough. 
\end{thm} 

\medskip\textbf{Remarks on $C^1$-small perturbations of the smoothed equation}. 
If $ m \neq 1$, it also follows  from  standard $C^1$-perturbation theory for hyperbolic fixed points that,  for equations  with data $ C^1$-close to $\tilde a $ and $ \tilde f$, there will exist a  unique nearby  periodic solution (the initial value of which is a fixed point of the perturbed period map) with analogous stability properties. We do not prove this in detail;  one would need a non-autonomous (here: periodic) statement of the type `nearby equations generate $ C^1-$ close  period  maps, with nearby fixed points', i.e., 
a non-autonomous version  of Theorem 1.8,  p.  20  from \cite{L-W95}. 
The robustness of the fixed point  of the period map $P$ (and of  the spectral properties
 with of $M = DP(\varphi)$) are then a very special case of  perturbation arguments for  hyperbolic sets for  non-invertible mappings, e.g., Theorem 6.6 from \cite{L-W95.2}. 

\medskip \textbf{Remarks on  the double period solutions.} Finally, we sketch  how analogous considerations apply to the solutions with double period from Section \ref{subsec3.2}: In this case, for  the smoothed equation  the period map $P$ on positive initial functions corresponds to a  one-dimensional map $ \tilde \Phi_1$  which is $C^1$-close to $\Phi_1$. One has to prove then  that  $P$ on  negative initial functions  corresponds to a  one-dimensional map $ \tilde \Phi_2$  which is $C^1$-close to $\Phi_2$. (This is not immediate from the first part, since the property that solutions 
of smooth and non-smooth equation coincide except for  small intervals gets lost after the second crossing of  zero.) 
Then  the second iterate $  P^2 = P(2T, T) \circ P(T, 0) =  P(T, 0) \circ P(T, 0) $  corresponds to $\tilde \Phi_2 \circ \tilde \Phi_1$, which is $C^1$ close 
to $ \Phi_2 \circ \Phi_1 $.  Now, for general $h$, one has in   view of  
(\ref{x3}):
$$ (\Phi_2 \circ \Phi_1)(h) = \Phi_2(kh + d) = -\Phi_1(-(kh + d)) = -[k\cdot(-(kh + d)) + d] = k^2h +kd-d. $$ 
Note that $h_*$ is a fixed point of $ \Phi_2\circ \Phi_1$, and the derivative there equals $k^2$. 
Then  existence and stability properties of this  fixed point transfer to the map 
$P^2$ for  the smoothed equation as described for the single period solutions   in Section \ref{Sec4}.

\medskip The above  considerations are  similar to the ones  in connection with  the dynamics generated by  autonomous delay equations with smoothed step functions,  as studied  in \cite{Wal81b}, \cite{HaleLin}, and \cite{L-W95}.
In particular, the  argument  here leading to Theorem \ref{Thm46} is similar to  the one from \cite{L-W95}, p. 31-34.

%%
%%
%% Section 5: Numerical Justification
%%

\begin{acknowledgement}
This collaborative research project was initiated during A.I.'s visit and stay at the Justus-Liebig Universit\"{a}t in Giessen, Germany, 
under a support from the Alexander von Humboldt Stiftung (June-August 2023).  It was further advanced at the mathematical research 
institute MATRIX in Australia during the workshop ``Delay Differential Equations and Their Applications'' held in December 2023
(https://www.matrix-inst.org.au/events/delay-differential-equations-and-their-applications/). 
\end{acknowledgement}

%%%%%%%%%%%%%%%%%%%%%%%% references.tex %%%%%%%%%%%%%%%%%%%%%%%%%%%%%%
% sample references
% %
% Use this file as a template for your own input.
%
%%%%%%%%%%%%%%%%%%%%%%%% Springer-Verlag %%%%%%%%%%%%%%%%%%%%%%%%%%
%
% BibTeX users please use
% \bibliographystyle{}
% \bibliography{}
%
%\biblstarthook{

%\reftitle{References}

\end{document}